\documentclass[12pt,a4paper]{article}
\usepackage{amsfonts}
\usepackage{mathrsfs}
\usepackage{amsmath,amssymb,amsthm}
\usepackage{hyperref}
\usepackage{cleveref}
\usepackage{tikz}
\usepackage{setspace}
\usetikzlibrary{positioning, shapes.geometric}
\usepackage{geometry}
\newtheorem{theorem}{Theorem}[section] 
\newtheorem{definition}[theorem]{Definition} 
\newtheorem{lemma}[theorem]{Lemma}  
  
\newtheorem{claim}[theorem]{Claim} 
 
\newtheorem{conjecture}[theorem]{Conjecture}

\numberwithin{equation}{section}

\usepackage{titlesec}
\titleformat{\section}[block]
{\normalfont\fontsize{16}{19.2}\selectfont\bfseries}
{\thesection}
{1em}
{}

\title{\textbf{On the product of cross-intersecting families with maximal covering number}}
\author{
	Long Lin\textsuperscript{1} , Peter Frankl\textsuperscript{2} 
	 , Hehui Wu\textsuperscript{1}\\ \\
	\textsuperscript{1}Shanghai Center for Mathematical Sciences\\ Fudan University, Shanghai, China\\ \\
	\textsuperscript{2}R\'enyi Institute, Budapest, Hungary
}
\date{}

\begin{document}
	
	\maketitle
	
	\begin{abstract}
		For  integers $k  ,\ell \geq 2$ let $m(k,\ell)$ denote the maximum of $|\mathcal{F}| |\mathcal{G}|$ where the maximum is taken over all pairs of cross-intersecting families, $\mathcal{F}$ being a $k$-graph with covering number $\ell$ and $\mathcal{G}$ a $\ell$-graph with covering number $k$ (see the paper for the definitions). Erdos and Lovasz initiated the study of the one family version. That is, they provided lower and upper bounds on the maximal size $m(k)=|\mathcal{F}|$ where $\mathcal{F}$ is an intersecting k-graph with covering number $k$.
		In many similar situations  $m(k,k)=m(k)^2$ holds. However, as our results show $m(k,k)/m(k)^2$ is tending to infinity as $k$ grows(Th.1.5) . For $k>k_0$ we establish the exact value $m(k,k)=(k^{k-1}+k-1)^2$(Th.1.6). As to smaller values we prove $m(3,3)=121$ (Th.1.7)  and determine $m(2,k) $ for all $k\geq 2$ (Th.1.8).
	\end{abstract}

	\section{Introduction}
	
	Let $[n]=\{1,2,\ldots,n\}$ be the standard $n$-element set, $\binom{[n]}{k}$ the collection of its $k$-subsets. Subsets of $\binom{[n]}{k}$ are called $k$-graphs.
	
	\begin{definition}
		A $k$-graph $\mathcal{F}$ is called \emph{intersecting} if $F \cap F' \neq \varnothing$ holds for all $F, F' \in \mathcal{F}$. Similarly, if $\mathcal{F}$ is a $k$-graph and $\mathcal{G}$ is an $\ell$-graph then they are called \emph{cross-intersecting} (CI for short) if $F \cap G \neq \varnothing$ holds for all choices of $F \in \mathcal{F}$, $G \in \mathcal{G}$.
	\end{definition}

	Erdős, Ko and Rado \cite{EKR} proved that
	\begin{equation}
		|\mathcal{F}| \leq \binom{n-1}{k-1}
	\end{equation}
	holds if $\mathcal{F}$ is an intersecting $k$-graph and $n \geq 2k$.
	
	For $n > 2k$ the only family that achieves equality in (1.1) is the full star $\mathcal{S}_i := \{S \in \binom{[n]}{k} : i \in S\}$. Pyber \cite{P} proved a product version of (1.1) which was sharpened by Matsumoto et al.
	
	\begin{theorem}[\cite{P}, \cite{MT}]
		Let $k \geq \ell$ be positive integers, $n \geq 2k$. Suppose that $\mathcal{F} \subseteq \binom{[n]}{k}$ and $\mathcal{G} \subseteq \binom{[n]}{\ell}$ are cross-intersecting. Then
		\begin{equation}
			|\mathcal{F}||\mathcal{G}| \leq \binom{n-1}{k-1}\binom{n-1}{\ell-1}.
		\end{equation}
		Obviously (1.1) is a consequence of (1.2).
	\end{theorem}
	
	Let us recall two important notions for a $k$-graph $\mathcal{F}$. The \emph{matching number} $\nu(\mathcal{F})$ is the maximum number of pairwise disjoint members (edges) of $\mathcal{F}$. The \emph{covering number} $\tau(\mathcal{F})$ is the minimum size of a set $T$ satisfying $T \cap F \neq \varnothing$ for all $F \in \mathcal{F}$. Such a $T$ is called a \emph{transversal}.
	
	Note that $1 \leq \nu(\mathcal{F}) \leq \tau(\mathcal{F})  \leq k$ holds for non-empty intersecting $k$-graphs.
	
	In a celebrated paper Erdős and Lovász \cite{EL} proved (among other things) that
	\begin{equation}
		|\mathcal{F}| \leq k^k
	\end{equation}
	holds for any intersecting $k$-graph with $\tau(\mathcal{F}) = k$. The scope of the present note is to consider the cross-intersecting version of this problem.
	
	\begin{definition}
		Let $m(k) = \max\{|\mathcal{F}| : \mathcal{F} \text{ is an intersecting } k\text{-graph with } \tau(\mathcal{F}) = k\}$.
		For $k , \ell \geq 1$ define 
		
		$m(k, \ell) = \max\{|\mathcal{F}||\mathcal{G}| : \mathcal{F} \text{ is a } k\text{-graph, } \mathcal{G} \text{ is an } \ell\text{-graph, they are cross-intersecting} \\ \text{and satisfy}$  $ \tau(\mathcal{F}) =  \ell, \tau(\mathcal{G}) =k \}$.
	\end{definition}

    Note that $m(k,\ell)=m(\ell ,k )$ and $m(k,1)=k$.
	
	The exact or even the asymptotic determination of $m(k)$ seems to be an untreatable problem. For $k=2$ and $3$, $m(2)=3$ and $m(3)=10$ are known and in both cases $\binom{[2k-1]}{k}$ provides equality, although for $k=3$  Erdős and Lovász discovered a completely different $3$-graph also attaining the equality. Already for $k=4$, there is a large gap (cf. \cite{FOT} and \cite{FW}):
	
	\begin{equation}
		42 \leq m(4) \leq 64.
	\end{equation}
	
	For a $k$-graph $\mathcal{F}$ with $\tau(\mathcal{F}) = t$, define $\mathcal{T}(\mathcal{F}) = \{T : |T| = t, T \text{ is a transversal of } \mathcal{F}\}$. 
	
	Let us recall a simple but important result of Gyárfás \cite{Gyarfas} :
	
	\begin{theorem}\label{thm:gyarfas}
		For any $k$-graph $\mathcal{F}$ with $\tau(\mathcal{F}) = t$,
		\begin{equation}
			|\mathcal{T}(\mathcal{F})| \leq k^t
		\end{equation}
		with equality if and only if $\mathcal{F}$ consists of $t$ pairwise disjoint edges.
	\end{theorem}
	
	We shall use this result together with a construction to show:
	
	\begin{theorem}\label{thm:main}
		Let $k , \ell \geq 2$ be integers. Then
		\begin{equation}
			(\ell^{k-1}+\ell-1) (k^{\ell-1}+k-1) \leq m(k, \ell) \leq \ell^k k^{\ell}.
		\end{equation}
	\end{theorem}

	Let us note that the ratio of the upper and the lower bound is less than $\ell k$ which is much smaller than the corresponding ratio in the case of $m(k)$  as $k \to \infty$.
	
     We  further improve the upper bound as follows:
	\begin{theorem}\label{}
	Let $k , \ell \geq 2$ be integers. Then 
	\begin{equation}
		m(k, \ell) \leq (1+\min\{k,\ell\}) k^{\ell-1} \ell^{k-1}
	\end{equation}
	Moreover if   $ \log k < \ell < e^{k}$ , then
		\begin{equation}
		m(k, \ell) = (\ell^{k-1}+\ell-1) (k^{\ell-1}+k-1) \text{ for } k,\ell \geq N .
		\end{equation}
	
\end{theorem}
	
	Let us note that  in particular $m(k,k)=(k^{k-1}+k-1)^2 $ for $k\geq k_0$ and if $\ell \geq 2$ is fixed, then (1.7) implies $m(k, \ell) =  O(k^{\ell-1}\ell^{k-1}) $ as $k \to \infty$.
	
	To determine the exact value of $m(k,\ell)$ is quite a challenge for small values of $k$ and $\ell$. The second part of the present paper is taken up by the proof of the following two results.

	\begin{theorem}\label{}
		\begin{equation}
			m(3,3) = 11^2 = 121.
		\end{equation}
	\end{theorem}

	\begin{theorem}\label{prop:seven}
		\begin{align}
			m(2,2) &= 9, \\
			m(3,2) &= 25, \\
		    m(k,2) &= k2^k(k\geq 4). 
		\end{align}
	\end{theorem}
	
	\section{Preparations}
	
	 Let us start with a simple observation.
	
	\begin{lemma}\label{lem:twoone} Suppose that $\mathcal{F}$ is a $k$-graph, $\mathcal{G}$ is a $\ell$-graph ,$\tau(\mathcal{F}) =  \ell, \tau(\mathcal{G}) =k $ and $\mathcal{F}$ , $\mathcal{G}$ are cross-intersecting.Then
		\[
		\bigcup_{F \in \mathcal{F}} F = \bigcup_{G \in \mathcal{G}} G.
		\]
	\end{lemma}
	
	\begin{proof}
		Suppose the contrary and assume by symmetry that there is a vertex $x$ with $x \in F \in \mathcal{F}$ which is not contained in any $G \in \mathcal{G}$. Then $F \cap G = (F \setminus \{x\}) \cap G$ implies that $F \setminus \{x\}$ is a transversal of $\mathcal{G}$. This contradicts $\tau(\mathcal{G}) = k$.
	\end{proof}

	\begin{definition}\label{def:optimalpair}
	For $k, \ell \geq 2$ and a $k$-graph $\mathcal{F}$, a $\ell$-graph $\mathcal{G}$, we say that $\mathcal{F}, \mathcal{G}$ form a $(k,\ell)$ \emph{optimal pair} if $\mathcal{F}$ and $\mathcal{G}$ are cross-intersecting, $ \tau(\mathcal{F}) =  \ell, \tau(\mathcal{G}) =k$ and moreover $|\mathcal{F}||\mathcal{G}| = m(k,\ell)$.
\end{definition}
In what follows, we always assume that $\mathcal{F}, \mathcal{G}$ form an optimal pair.
\begin{claim}\label{claim:optimalpair} Every optimal pair is maximal :
	$$\mathcal{F} = \mathcal{T}(\mathcal{G}),\mathcal{G} = \mathcal{T}(\mathcal{F}).$$
\end{claim}

\begin{proof}
	By symmetry we only prove the first equality. Since $\mathcal{F}, \mathcal{G}$ are cross-intersecting,  $\mathcal{F} \subset \mathcal{T}(\mathcal{G})$, this implies $|\tau(\mathcal{T}(\mathcal{G}))| \geq |\tau(\mathcal{F})|\geq \ell$. As every $G \in \mathcal{G}$ is a transversal of $\mathcal{T}(\mathcal{G})$, $\tau(\mathcal{T}(\mathcal{G})) = \ell$ follows. Consequently the pair $\mathcal{T}(\mathcal{G}), \mathcal{G}$ would contradict the optimality of the pair $\mathcal{F}, \mathcal{G}$ unless  $\mathcal{T}(\mathcal{G}) = \mathcal{F}$.
\end{proof}

	Let us fix some more notation. Let $X=	\bigcup_{F \in \mathcal{F}} F$ and for sets $S, U\subset X$ define $\mathcal{F}(S) = \{F \setminus S : S \subset F \in \mathcal{F}\}$, $\mathcal{F}(\widehat{S}) = \{F : S \subset F \in \mathcal{F}\}$,  $\mathcal{F}(\overline{S}) = \{F : F \cap S = \emptyset\}$, $\mathcal{F}(\overline{S};\widehat{U})=(\mathcal{F}(\overline{S}))(\widehat{U})$ and $\mathcal{F}(\overline{S};U)=(\mathcal{F}(\overline{S}))(U)$. If $S = \{x\}$ we simply write $\mathcal{F}(x)$, $\mathcal{F}(\widehat{x})$ and $\mathcal{F}(\overline{x})$.Note that  $|\mathcal{F}(x)|=|\mathcal{F}(\widehat{x})|$.
	
		For a hypergraph $\mathcal{H}$ , let $\Delta(\mathcal{H})$ denote its maximal degree.
	\begin{lemma}\label{2.4} Let $\widetilde{\mathcal{F}} \subset \mathcal{F}, S \subset X$ with $|S|\in [0,k-1]$, then there exists $x\in X-S$ such that $|\widetilde{\mathcal{F}}(S\cup \{x\})|\geq \ell^{-1} |\widetilde{\mathcal{F}}(S)|$.  In particular, $|\mathcal{F}|\leq \ell \Delta(\mathcal{F}) .$
	\end{lemma}

\begin{proof} 
	Since $\tau(\mathcal{G})=k>|S|$, there exists $G\in \mathcal{G}$ such that $G \cap S =\emptyset $ . By the CI property, $\widetilde{\mathcal{F}}(\widehat{S})\subset \displaystyle \bigcup_{y\in G}\widetilde{\mathcal{F}}(\widehat{S\cup \{y\}})$, this implies $|\widetilde{\mathcal{F}}(S)|\leq \displaystyle \sum_{y\in G}|\widetilde{\mathcal{F}}(S\cup \{y\})|$. 
\end{proof}

	 Note that  $\Delta(\mathcal{H})=1$ if and only if $\mathcal{H}$ consists of pairwise disjoint edges. If $\Delta(\mathcal{F})=1$ , by $\tau(\mathcal{F})=\ell$, $|\mathcal{F}|=\ell$ and $|\mathcal{G}|=k^{\ell}$, then Theorems 1.5, 1.6 ,1.7 and 1.8 are valid. If $\Delta(\mathcal{G})=1$ , the same holds. So throughout the following, we always assume that
	\begin{equation}
\min\{\Delta(\mathcal{F}),\Delta(\mathcal{G})\}\geq 2.
\end{equation} 

\begin{lemma}\label{2.5} If $|S|< k $ and $\mathcal{F}(S)\neq \emptyset$, then $ \tau(\mathcal{G}(\overline{S}))= k-|S|$ and $\mathcal{F}(S)= \mathcal{T}(\mathcal{G}(\overline{S})$; \\If $|S|< \ell $ and $\mathcal{G}(S)\neq \emptyset$, then $ \tau(\mathcal{F}(\overline{S}))= \ell-|S|$ and $\mathcal{G}(S)= \mathcal{T}(\mathcal{F}(\overline{S})$.
\end{lemma}

\begin{proof}
		By symmetry we only prove the first conclusion. Since $|S|<k$ and $\tau(\mathcal{G})=k$ , $\mathcal{G}(\overline{S})\neq \emptyset$ follows. For every $T \in \mathcal{T}(\mathcal{G}(\overline{S})$, $T\cup S$ is a transversal of $\mathcal{G}$. Hence $k=\tau(\mathcal{G})\leq |T\cup S|\leq |T|+|S|$ i.e. $|T|\geq k-|S|$. This implies $\tau(\mathcal{G}(\overline{S}))\geq k-|S|$. On the other hand, every  $A\in \mathcal{F}(S)\neq \emptyset$, by the CI property, is a transversal of $\mathcal{G}(\overline{S})$ whence $\tau(\mathcal{G}(\overline{S}))\leq |A|= k-|S|$ and $ \tau(\mathcal{G}(\overline{S}))= k-|S|$ follow. Hence, $A\in \mathcal{T}(\mathcal{G}(\overline{S})$ and  $T\cup S \in  \mathcal{T}(\mathcal{G})$ and $T\cap S=\emptyset$ . Using Claim 2.3,  $T\cup S \in  \mathcal{F}$ i.e. $ T\in  \mathcal{F}(S)$. So the conclusion holds.
\end{proof}

\begin{lemma}\label{2.6}
If $|S|\leq k $, then $|\mathcal{F}(S)| \leq {\ell}^{k-|S|}$ ; If $|S|\leq \ell $,then $|\mathcal{G}(S)| \leq {k}^{\ell-|S|}$.
\end{lemma}

\begin{proof}
	 	By symmetry we only prove the first inequality. If $|S|=k$ or $\mathcal{F}(S)=\emptyset$, then the statement is true. Suppose $|S|<k$ and $\mathcal{F}(S)\neq \emptyset$. Then using (1.5) and Lemma 2.5, $|\mathcal{F}(S)|= |\mathcal{T}(\mathcal{G}(\overline{S}))|\leq {\ell}^{k-|S|}$.
\end{proof}
    More generally , we have the following
\begin{lemma}\label{} If $|\mathcal{F}(\overline{V};U)|> (t-1)\ell^{k-1-|U|} , 0\leq t \leq \ell-|V|$ , then $|\mathcal{G}(\overline{U};V)|\leq \binom{k+t-1-|U|}{t}k^{\ell-|V|-t}$.
\end{lemma}

\begin{proof}
	For a set $G$ and an ordered set $F=\{x_1,x_2,...,x_k\}$ define  
	$$h_G(F)=\min\{j:x_j\in G\}, F[i]=\{x_1,...,x_i\} \text{ and } F(i)=x_i \text{ for } i\in [k].$$
	Claim: if $|\mathcal{F}(\overline{V};U)|> (t-1)\ell^{k-1-|U|} , 0\leq t \leq \ell-|V|$  then there exists an injection $w:\mathcal{G}(\overline{U};\widehat{V}) \to [k]^{\ell-|V|}$ such that $|U|<w_1\leq w_2 \leq ...\leq w_t$.  
	
   Let $v=|V|, u=|U|$ and let us prove the claim using induction on $t$.	When $t=0$ , by Lemma 2.6 , we are done. Suppose $1\leq t  \leq \ell-v$ and  the claim holds for $t-1$.

	 Define the (ordered) set $F_1=\{x_1,x_2,...,x_k\}\in \mathcal{F}(\overline{V})$ by a greedy algorithm which is based on Lemma 2.4. First set $F_1[u]=\{x_1,...,x_u\}=U$  and choose $x_{u+1}$ to satisfy $|\mathcal{F}(\overline{V};F_1[u+1]) |=|\mathcal{F}(\overline{V}; x_1,...,x_{u+1}) |\geq \ell^{-1}|\mathcal{F}(\overline{V};F_1[u]) |$ . Then fix  $x_{u+2}$ to satisfy $|\mathcal{F}(\overline{V}; F_1[u+2]) |\geq \ell^{-1}|\mathcal{F}(\overline{V};F_1[u+1]) |$ and so on.
	 
	By the CI property define $w_1(G)=h_G(F_1)\in [k]$ for each $G\in \mathcal{G}(\overline{U};\widehat{V}) $. Clearly , $w_1(G)>u$. Note that for each $u< i \leq k$ , we have 
	$$|\mathcal{F}\left(\overline{V\cup \{F_1(i)\}};F_1[i-1] \right)|>(t-2)\ell^{k-i}$$
	because
	$$|\mathcal{F}\left(\overline{V}; F_1[i-1]\right)|>(t-1)\ell^{k-i} \text{ and } |\mathcal{F}\left(\overline{V}; F_1[i]\right)|\leq \ell^{k-i}.$$
	
	Therefore by the induction hypothesis , there exists an injection $$(w^i_2,...,w^i_{\ell-v}):\mathcal{G}\left(\overline{F_1[i-1]};\widehat{V\cup \{F_1(i)\}}\right) \to [k]^{\ell-v-1}$$ such that $i\leq w^i_2(G) \leq ...\leq w^i_t(G)$  for $G\in \mathcal{G}\left(\overline{F_1[i-1]};\widehat{V\cup \{F_1(i)\}}\right)  $.  
	Since $\mathcal{G}(\overline{U};\widehat{V})$  is the disjoint union of $\mathcal{G}\left(\overline{F_1[i-1]};\widehat{V\cup \{F_1(i)\}}\right),u<i\leq k$ , we can define  $(w_2,...,w_{\ell-v}):\mathcal{G}(\overline{U};\widehat{V}) \to [k]^{\ell-v-1}$ such that $(w_1,...,w_{\ell-v}) $ is injective and $u<w_1\leq w_2 \leq ...\leq w_t$.  

    So $|\mathcal{G}(x)|\leq \binom{k+t-1-u}{t}k^{\ell-v-t}$, completing the proof.
\end{proof}

For a $h$-graph $\mathcal{H}$ define 
$$N_s(\mathcal{H}) = |\{T : |T| = s, T \text{ is a minimal(for containment) transversal of } \mathcal{H}\}|.$$

\begin{lemma}\label{} Suppose $x\in X$,then 
		\[
		 |\mathcal{F}(\overline{x})|\leq \displaystyle \sum_{s=1}^{k}{\ell }^{k-s} N_s(\mathcal{G}(x)),
		 \]
		 \[
   |\mathcal{G}(\overline{x})|\leq \displaystyle \sum_{s=1}^{\ell}k^{\ell-s} N_s(\mathcal{F}(x)).
	\]
	
\end{lemma}
\begin{proof}
	By symmetry we only prove the second inequality. Using double counting for the number of ordered pairs $(B,T)$ , where  $B\in \mathcal{G}(\overline{x})$ and $T \text{ is a minimal transversal of } \mathcal{F}(x)$ of size at most $\ell$. By cross-intersecting property every $B\in \mathcal{G}(\overline{x})$ contains a transversal whence also a minimal transversal of $\mathcal{F}(x)$. On the other hand, using Lemma 2.6 , every minimal transversal of $\mathcal{F}(x)$ of size $s\leq \ell$ is contained in at most $k^{\ell-s}$ members of $\mathcal{G}$. By $\mathcal{G}(\overline{x})\subset \mathcal{G}$ the inequality follows.
\end{proof}

	\section{Proof of Theorem 1.5}
	
	Let $\mathcal{F}, \mathcal{G}$ form an optimal pair. Using (1.5) and Claim  2.3, $|\mathcal{F}| = |\mathcal{T}(\mathcal{G})| \leq \ell^k$ and $|\mathcal{G}| = |\mathcal{T}(\mathcal{F})| \leq k^\ell$, implying the upper bound.
	
	To prove the lower bound we give a construction. For pairwise disjoint sets $Y_1, \ldots, Y_r$ define 
	\[
	Y_1 \times \cdots \times Y_r = \{ \{y_1, \ldots, y_r\} : y_i \in Y_i, 1 \leq i \leq r \}.
	\]
	
	Fix $k-1$ pairwise disjoint $\ell$-sets $A_1, \ldots, A_{k-1}$ and $\ell-1$ pairwise disjoint $k$-sets $B_1, \ldots, B_{\ell-1}$ so that $A_i \cap B_j \neq \varnothing$ for all $1 \leq i \leq k-1$, $1 \leq j \leq \ell-1$.
	
	(This is always possible. E.g., one considers $(k-1) \times (\ell-1)$ vertices arranged in a grid. Let $\widetilde{A}_1, \ldots, \widetilde{A}_{k-1}$ be the rows and $\widetilde{B}_1, \ldots, \widetilde{B}_{\ell-1}$ be the columns. Then add an extra vertex to each of the $(k-1)+(\ell-1)$ sets without destroying the required disjointness.)
	
	Fix an extra vertex $y$. Set $	\mathcal{A} = \{A_1, \ldots, A_{k-1}\}$, $	\mathcal{B} = \{B_1, \ldots, B_{\ell-1}\}$ and define
	\[
	\mathcal{F} = 	\mathcal{B} \cup (A_1 \times \cdots \times A_{k-1} \times \{y\}), \quad
	\mathcal{G} = 	\mathcal{A} \cup (B_1 \times \cdots \times B_{\ell-1} \times \{y\}).
	\]
	
	It should be clear that $\mathcal{F}$ is a $k$-graph, $\mathcal{G}$ is a $\ell$-graph, $|\mathcal{F}| =\ell^{k-1}+\ell-1$, $|\mathcal{G}| = k^{\ell-1}+k-1$, and $\mathcal{F}, \mathcal{G}$ are cross-intersecting. We need to prove $\tau(\mathcal{F}) = \ell$ and $\tau(\mathcal{G}) = k$.
	
	By symmetry let us prove it for $\mathcal{F}$. Let $T$ be a transversal of $\mathcal{F}$. There are two cases to consider.
	
	If $y \in T$, then $T$ intersects each of the $\ell$ pairwise disjoint sets $B_1, \ldots, B_{\ell-1}$ and $\{y\}$. Hence $|T| \geq \ell$.
	
	If $y \notin T$, then $T \cap \{a_1, \ldots, a_{k-1}\} \neq \varnothing$ for all $\ell^{k-1}$ edges of $A_1 \times \cdots \times A_{k-1}$, which forces that $A_j\subset T$ for some $1\leq j<k$. Hence $|T| \geq \ell$.
	
	Thus $\mathcal{F}, \mathcal{G}$ satisfy all requirements, concluding the proof. 
	
	\section{Proof of Theorem 1.6}
Let $\mathcal{F}, \mathcal{G}$ be a $(k,\ell)$ optimal pair, then by Theorem 1.5 , $|\mathcal{F}||\mathcal{G}|>\ell^{k-1}k^{\ell-1}$.

 The next lemma is inspired by Lemma 2 of \cite{Arman} .
	\begin{lemma}
		If $ |\mathcal{F}(\overline{S})|\geq (t-1)\Delta(\mathcal{F})+ \ell^{2k/3+1}$ , $t \leq \ell -|S|-1$ , then 
		$$|\mathcal{G}(S)|\leq (\frac{2\sqrt{2}}{3})^t k^{\ell-|S|}.$$
	\end{lemma}

\begin{proof}
		Claim: if $ |\mathcal{F}(\overline{S})|\geq (t-1)\Delta(\mathcal{F})+ \ell^{2k/3+1}$ , $t \leq \ell -|S|-1$ then there exists an injection $w:\mathcal{G}(S) \to [k]^{\ell-|S|}$ such that for each odd number $1\leq i \leq t$, if $w_i>\frac{2k}{3}$ then $w_{i+1}>1+k/3$.
		
		 Let $s=|S|$ and let us prove the claim using induction on $t$.	When $t\leq 0$ , by Lemma 2.6 , we are done. Suppose $1\leq t \leq \ell -|S|-1$ and  the claim holds for $\leq t-1$. 
		 
		 Since $|\mathcal{F}(\overline{S})|\geq \ell^{2k/3+1}$, by Lemma 2.4 we can define the $x_j,1\leq j\leq 1+k/3$ to satisfy   $|\mathcal{F}(\overline{S};x_1,...,x_j) |\geq \ell^{-1}|\mathcal{F}(\overline{S};x_1,...,x_{j-1})|$.
		 
		Let $P:= \bigcap \mathcal{F}(\overline{S};x_1,...,x_{\lfloor k/3 \rfloor +1})$, then 
		$$\ell^{1+2k/3-\lfloor k/3 \rfloor-1} \leq  |\mathcal{F}(\overline{S};x_1,...,x_{\lfloor k/3 \rfloor +1})|=|\mathcal{F}(\overline{S};\{x_1,...,x_{\lfloor k/3 \rfloor +1}\}\cup P)|\leq \ell^{k-\lfloor k/3 \rfloor-1-|P|}$$
		hence $|P|\leq k/3 -1$. So we can define an ordered set $F_1=\{x_1,...,x_k\}$ such that $P\subset F_1[\lfloor 2k/3 \rfloor]$. Define $w_1(G)=h_G(F_1)$ for $G\in \mathcal{G}(S)$.
		
		For any $i > \frac{2k}{3}$ , $x_i\notin P$ . By definition , we can fix a $F^{i}_2\in  \mathcal{F}(\overline{S\cup \{x_i\}})$  such that  $F^{i}_2[\lfloor k/3 \rfloor +1]=F_1[\lfloor k/3 \rfloor +1]$ . For any $i\leq \frac{2k}{3}$, we simply fix a $F^{i}_2\in  \mathcal{F}(\overline{S\cup \{x_i\}})$.	Define $w_2(G)=h_G(F^{w_1(G)}_2)\in [k]$ for each $G\in \mathcal{G}$. Clearly, if $w_1(G)>\frac{2k}{3}$ then $w_{2}(G)>1+k/3$.
		
		For each pair $(i,j)$ which satisfies if $i>\frac{2k}{3}$ then $j>1+k/3$, we have 
		$$|\mathcal{F}(\overline{S\cup\{F_1(i),F^i_2(j)\}})|\geq |\mathcal{F}(\overline{S})|-2\Delta(\mathcal{F})\geq (t-3)\Delta(\mathcal{F})+\ell^{2k/3+1}.$$
		Therefore by the induction hypothesis , there exists an injection $$(w^{(i,j)}_3,...,w^{(i,j)}_{\ell-s}):\mathcal{G}(\widehat{S\cup \{F_1(i),F^i_2(j)\}})\to [k]^{\ell-s-2}$$
		such that for each odd number $3\leq p \leq t$, if $w^{(i,j)}_p(G)>\frac{2k}{3}$ then $w^{(i,j)}_{p+1}(G)>1+k/3$ .
		
		Since $\mathcal{G}(\widehat{S})$  is the disjoint union of $\mathcal{G}(\widehat{S\cup \{F_1(i),F^i_2(j)\}})$  , we can define  $(w_3,...,w_{\ell-s}):\mathcal{G}(\widehat{S}) \to [k]^{\ell-s-2}$ such that $(w_1,...,w_{\ell-s}) $ is injective and for each odd number $ p \in [t]$, if $w_p(G)>\frac{2k}{3}$ then $w_{p+1}(G)>1+k/3$ .
		
		So $|\mathcal{G}(S)|\leq  (\frac{8}{9}k^2)^{\lceil t/2 \rceil}k^{\ell-s-2\lceil t/2 \rceil}\leq  (\frac{2\sqrt{2}}{3})^tk^{\ell-s}$, completing the proof.
\end{proof}

	The following improvement of Gyárfás Theorem is kind of folklore.

\begin{lemma}
	For any $k$-graph $\mathcal{F}$ with $\tau(\mathcal{F}) = t$ and $|\mathcal{F}|>t$,
	\begin{equation}
		|\mathcal{T}(\mathcal{F})| \leq (k^2-k+1)k^{t-2}
	\end{equation}
\end{lemma}

\begin{proof}
	Since $|\mathcal{F}|>t , \tau(\mathcal{F}) = t$ , we may fix $F, F' \in \mathcal{F}$  with $1\leq |F\cap F'| <k $. Set $|F\cap F'|=s$ and  $F=\{x_1,..,x_s,y_{1},...,y_{k-s}\} , F'=\{x_1,..,x_s,z_{1},...,z_{k-s}\}$ . Note that for any set $H$ intersecting both $F$ and $F'$ either $x_i\in H$ for some $1\leq i\leq s$ or $H$ contains (at least) one of the pairs $\{y_i,z_j\}, 1\leq i,j \leq k-s$.
	
	By (1.5), $|\mathcal{T}(\mathcal{F})(x_i)|\leq k^{t-1}$ and $|\mathcal{T}(\mathcal{F})(y_i,z_j)|\leq k^{t-2}$.
	Hence $|\mathcal{T}(\mathcal{F})|\leq sk^{t-1}+(k-s)^2k^{t-2} =(k^2-ks+s^2)k^{t-2}$. As $k^2-ks+s^2\leq k^2-k+1$ for $1\leq s\leq k-1$, (4.1) follows.
\end{proof}
     
     Now we use Lemmas 2.7 , 4.1 and 4.2 to prove Theorem 1.6:
\begin{proof}
	(1) Suppose $(t-1)\ell^{k-1}<|\mathcal{F}|\leq t\ell^{k-1}, t \in \mathbb{Z}_+$, by Lemma 2.7, $|\mathcal{G}|\leq \binom{k+t-1}{t}k^{\ell-t}$. Hence $|\mathcal{F}||\mathcal{G}|\leq  t\ell^{k-1}\binom{k+t-1}{t}k^{\ell-t}=\ell^{k-1}k^{\ell}[t\binom{k+t-1}{t}k^{-t}]\leq \ell^{k-1}k^{\ell}(1+\frac{1}{k})=\ell^{k-1}k^{\ell-1}(1+k)$, because $t\binom{k+t-1}{t}k^{-t}=\frac{(1+\frac{t-1}{k})(1+\frac{t-2}{k})...(1+\frac{1}{k})}{(t-1)!}=(1+\frac{1}{k}) \displaystyle \prod_{i=2}^{t-1}\frac{1+\frac{i}{k}}{i} \leq 1+\frac{1}{k}$. Similarly, $|\mathcal{F}||\mathcal{G}|\leq \ell^{k-1}k^{\ell-1}(1+\ell)$. This proves (1.7).
	
	(2)	Suppose that $\log k <\ell <e^k$ which will be used in Claim 4.5.
	\begin{claim}\label{} 
		$\Delta(\mathcal{F})> \ell^{1+2k/3}, \Delta(\mathcal{G}) > k^{1+2\ell/3}$ 
	\end{claim}
	\begin{proof}
		Suppose $\Delta(\mathcal{F})\leq \ell^{1+2k/3}$ , then by Lemma 2.4,  $|\mathcal{F}|\leq \ell \Delta(\mathcal{F})\leq \ell^{2+2k/3}$. So $|\mathcal{F}||\mathcal{G}|\leq \ell^{2+2k/3} k^\ell <\ell^{k-1}k^{\ell-1}$ for $k\geq N$, contradiction.
	\end{proof}

   Define $y=f(x)$ by the equation $y(\log y -\log 2e)=3 \log x , x\geq 1$. Clearly $\frac{y}{\log x}\to 0$ as $x\to \infty$ .

    \begin{claim}\label{} 
    	$\Delta(\mathcal{F})\geq  f(k) \ell^{k-2}$ or   $\Delta(\mathcal{G})\geq  f(\ell) k^{\ell-2}$ .
    \end{claim}

\begin{proof}
	Suppose $\Delta(\mathcal{F})<  f(k)\ell^{k-2}$ and $\Delta(\mathcal{G})<  f(\ell)k^{\ell-2}$ . Let $t=\lceil |\mathcal{F}|/\Delta(\mathcal{F}) \rceil-1$ then $t\leq \ell-1$ and $t\geq |\mathcal{F}|/\Delta(\mathcal{F}) -1$ .Using Lemma 4.1 , $|\mathcal{G}|\leq (\frac{2\sqrt{2}}{3})^t k^{\ell}\leq (\frac{2\sqrt{2}}{3})^{ |\mathcal{F}|/\Delta(\mathcal{F}) -1} k^{\ell}$. Similarly, $|\mathcal{F}|\leq  (\frac{2\sqrt{2}}{3})^{ |\mathcal{G}|/\Delta(\mathcal{G}) -1} \ell^{k}$. So
	$$ \frac{|\mathcal{F}||\mathcal{G}|}{k^{\ell}\ell^{k}}\leq (\frac{2\sqrt{2}}{3})^{ |\mathcal{F}|/\Delta(\mathcal{F})+|\mathcal{G}|/\Delta(\mathcal{G}) -2} \leq \left(\frac{8}{9}\right)^{\sqrt{\frac{|\mathcal{F}||\mathcal{G}|}{\Delta(\mathcal{F})\Delta(\mathcal{G})}}-1}\leq \left(\frac{8}{9}\right)^{\sqrt{\frac{k\ell}{f(k)f(\ell)}}-1}.$$
Therefore $|\mathcal{F}||\mathcal{G}|\leq k^{\ell-1}\ell^{k-1}$ for $k,\ell\geq N$, contradiction.
\end{proof}

  By symmetry we may assume $|\mathcal{F}(x)|=\Delta(\mathcal{F})\geq  f(k) \ell^{k-2}$ .

 \begin{claim}\label{} 
   $|\mathcal{G}(x)|\geq  f(\ell) k^{\ell-2}$ .
\end{claim}
	\begin{proof}
		Suppose  $|\mathcal{G}(x)|<  f(\ell) k^{\ell-2}$.   Define $s\in \mathbb{Z}_+$ by $(s-1)\ell^{k-2}< |\mathcal{F}(x)|\leq s\ell^{k-2}$, then $\ell \geq s\geq f(k)$. Using Lemma 2.7 , $|\mathcal{G}(\overline{x})|\leq \binom{k+s-1}{s}k^{\ell-s}$.
		Suppose $(t-1)\ell^{k-1}<|\mathcal{F}(\overline{x})|\leq t\ell^{k-1}, t \in \mathbb{Z}_+$ . By Claim 4.3, $|\mathcal{G}|>\Delta(\mathcal{G}) > k^{1+2\ell/3}$. Using Lemma 4.1, $|\mathcal{F}|\leq (\frac{2\sqrt{2}}{3})\ell^{k}$, $t<\ell$ follows . Using Lemma 2.7, $|\mathcal{G}(x)|\leq \binom{k+t-1}{t}k^{\ell-1-t}$. 
		
		\begin{claim}\label{} 
			
		(i) Let $s_0=\lceil f(k) \rceil$, then $\binom{k+s_0-1}{s_0}k^{-s_0} \leq k^{-3} $ ,
		
		(ii) $s\binom{k+s-1}{s}k^{-s}\leq k^{-2}$ ,
		
		(iii) $t^2\binom{k+t-1}{t}k^{-t}<8$ .
		\end{claim}
	
	\begin{proof}
		(i) As $\binom{k+s_0-1}{s_0}k^{-s_0}\leq (\frac{e(k+s_0-1)}{s_0})^{s_0}k^{-s_0} \leq (\frac{2e}{s_0})^{s_0}\leq (\frac{2e}{f(k)})^{f(k)}$ and by definition $(\frac{2e}{f(k)})^{f(k)}=  k^{-3}$, (i) follows.
		
		(ii) Using $s\geq s_0$,  $s\binom{k+s-1}{s}k^{-s}=\frac{(1+\frac{s-1}{k})(1+\frac{s-2}{k})...(1+\frac{1}{k})}{(s-1)!}=\displaystyle \prod_{i=1}^{s-1}\frac{1+\frac{i}{k}}{i}\leq \prod_{i=1}^{s_0-1}\frac{1+\frac{i}{k}}{i}\leq s_0\binom{k+s_0-1}{s_0}k^{-s_0}\leq k^{-2}$.
		
		(iii) $t^2\binom{k+t-1}{t}k^{-t}=\frac{t(1+\frac{t-1}{k})(1+\frac{t-2}{k})...(1+\frac{1}{k})}{(t-1)!}=\frac{t}{t-1} \displaystyle \prod_{i=2}^{t-1}\frac{1+\frac{i}{k}}{i-1}(1+\frac{1}{k}) \leq \frac{t}{t-1} \displaystyle \prod_{i=2}^{3}\frac{1+\frac{i}{k}}{i-1}(1+\frac{1}{k})  \leq \displaystyle \prod_{i=1}^{3}(1+\frac{i}{k})<8 $ for $t\geq 2$.
		, while $t^2\binom{k+t-1}{t}k^{-t}=1 $  for $t= 1$ .
	\end{proof}
		 Let us estimate separately the four terms of the expansion of $|\mathcal{F}||\mathcal{G}|=(|\mathcal{F}(x)|+|\mathcal{F}(\overline{x})|)(|\mathcal{G}(x)|+|\mathcal{G}(\overline{x})|)$. Using Claim 4.6,
		
		$|\mathcal{F}||\mathcal{G}(\overline{x})|\leq   s\ell^{k-1}\binom{k+s-1}{s}k^{\ell-s}\leq \ell^{k-1} k^{\ell-2}$,
		
		$|\mathcal{F}(\overline{x})||\mathcal{G}(x)|\leq  t\ell^{k-1}\binom{k+t-1}{t}k^{\ell-1-t}=\ell^{k-1}k^{\ell-1}[t\binom{k+t-1}{t}k^{-t}]$,
		
		$|\mathcal{F}(\overline{x})||\mathcal{G}(x)|\leq  t\ell^{k-1} f(\ell) k^{\ell-2}$,
		
		 Thus $|\mathcal{F}(\overline{x})||\mathcal{G}(x)|\leq \sqrt{ \ell^{k-1}k^{\ell-1}[t\binom{k+t-1}{t}k^{-t}] t\ell^{k-1} f(\ell) k^{\ell-2}} =\sqrt{\ell^{2k-2}k^{2\ell-3}f(\ell)} \sqrt{t^2\binom{k+t-1}{t}k^{-t}}\leq 2\sqrt{2} \sqrt{\ell^{2k-2}k^{2\ell-3}f(\ell)}$.
		
		While $|\mathcal{F}(x)||\mathcal{G}(x)|\leq  \ell^{k-1} f(\ell) k^{\ell-2} $. Together with all above ,	$|\mathcal{F}||\mathcal{G}|<k^{\ell-1}\ell^{k-1}$ for $k\geq N$, contradiction.
	\end{proof}

		\begin{claim}\label{} 
		$|\mathcal{F}(\overline{x})|=\ell-1$ and $ |\mathcal{G}(\overline{x})|=k-1$.
	\end{claim}
	
	\begin{proof}
	By Claims 4.5 and 4.6,  $|\mathcal{F}(\overline{x})|\leq \ell^{k-3}$ and $|\mathcal{G}(\overline{x})|\leq k^{\ell-3}$.
	
	Suppose $|\mathcal{F}(\overline{x})|\geq \ell$ and $|\mathcal{G}(\overline{x})|\geq k$ ,then by Lemma 2.5 and 4.2, $|\mathcal{G}(x)|\leq (k^2-k+1)k^{\ell-3}$ and  $|\mathcal{F}(x)|\leq (\ell^2-\ell+1)\ell^{k-3}$. Therefore $|\mathcal{F}|\leq (1-\frac{1}{\ell}+2(\frac{1}{\ell})^2)\ell^{k-1} \leq  \ell^{k-1}$. Similarly $|\mathcal{G}|\leq k^{\ell-1}$ , consequently $|\mathcal{F}||\mathcal{G}|\leq \ell^{k-1}k^{\ell-1}$, contradiction.
	
	Suppose $\ell\leq |\mathcal{F}(\overline{x})|< \ell^{2k/3+1}$ and $|\mathcal{G}(\overline{x})|= k-1$ ,then by Lemma 2.5 and 4.2, $|\mathcal{G}(x)|\leq (k^2-k+1)k^{\ell-3}$ and  $|\mathcal{F}(x)|= \ell^{k-1}$. Therefore $|\mathcal{G}|\leq (1-\frac{1}{k}+2(\frac{1}{k})^2)k^{\ell-1}$ and $|\mathcal{F}|\leq (1+\ell^{2-k/3})\ell^{k-1}$. Since $1-\frac{1}{k}+2(\frac{1}{k})^2\leq 1-\frac{1}{2k} (<\frac{1}{1+\frac{1}{2k}})$ and $\ell^{k/3-1}\geq 2k$  for $k\geq N$ ,  $|\mathcal{F}||\mathcal{G}|\leq \ell^{k-1}k^{\ell-1}$, contradiction.
	
	Suppose $ |\mathcal{F}(\overline{x})|\geq  \ell^{2k/3+1}$ and $|\mathcal{G}(\overline{x})|= k-1$ , the by Lemmas 2.5, 4.1 and 4.2, $|\mathcal{G}(x)|\leq (\frac{2\sqrt{2}}{3})k^{\ell-1}$ and  $|\mathcal{F}(x)|= \ell^{k-1}$. Therefore $|\mathcal{G}|\leq \frac{2\sqrt{2}}{3} k^{\ell-1}+k^{\ell-3}$ and $|\mathcal{F}|\leq \ell^{k-1}+\ell^{k-3}$. Since $(1+\ell^{-2})(\frac{2\sqrt{2}}{3}+k^{-2})<1$ for $k,\ell\geq N$ ,  $|\mathcal{F}||\mathcal{G}|\leq \ell^{k-1}k^{\ell-1}$ follows , contradiction.
	
	 Similarly for  $|\mathcal{F}(\overline{x})|= \ell-1$ and $|\mathcal{G}(\overline{x})|\geq k$.
	\end{proof}
	
	By Lemma 2.5 and Claim 4.7 ,  $\mathcal{F}(\overline{x})$ consists of $\ell-1$ pairwise disjoint edges and  $|\mathcal{G}(x)|=|\mathcal{T}(\mathcal{F}(\overline{x}))|=k^{\ell-1}$, hence $|\mathcal{G}|=k^{\ell-1}+k-1$. Similarly, $|\mathcal{F}|=\ell^{k-1}+\ell-1$, completing the proof.
\end{proof}
\noindent\textbf{Remark 4.8.}  If $\ell > e^k $, then by the proof, we can see that $m(k,\ell)\leq (2+2/k)\ell^{k-1}k^{\ell-1}$ for $k \geq N $.

	\section{Proof of Theorem 1.7}
	Let $\mathcal{F}$ and $ \mathcal{G}$ be $3$-graphs forming an optimal pair.
	\begin{claim}\label{} For all $x\in X$, $|\mathcal{F}(x)|\leq 9$; for $x\neq y$, $|\mathcal{F}(\{x,y\})|\leq 3$ and moreover if the equality holds then $\mathcal{G}(\overline{\{x,y\}})=\{\bigcup\mathcal{F}(\{x,y\})\}$.
	\end{claim}
\begin{proof}
	It is a straightforward corollary of Lemmas 2.6 and 2.5. 
\end{proof}

	\begin{claim}\label{} If $|\mathcal{F}(x)|=3$, then $|\mathcal{G}(\overline{x})|\leq 10$.
\end{claim}
		
\begin{proof}
	
	If $\Delta(\mathcal{F}(x))=3$, then $\mathcal{F}(x)=K_{1,3}$, hence $(N_s(\mathcal{F}(x)))_{s=1}^{3}=(1,0,1)$.
	
	If $\Delta(\mathcal{F}(x))=2$, then $\mathcal{F}(x)=C_3$ or $P_3$ or $P_2 + P_1$, the corresponding values of $(N_s(\mathcal{F}(x)))_{s=1}^{3}$ are  $(0,3,0)$ or $(0,3,0)$ or $(0,2,2)$.
	
	If $\Delta(\mathcal{F}(x))=1$, then $\mathcal{F}(x)=P_1+P_1+P_1$, hence $(N_s(\mathcal{F}(x)))_{s=1}^{3}=(0,0,8)$.
	
	Using Lemma 2.8, the conclusion follows.
	
\end{proof}

	\begin{claim}\label{} Suppose $|\mathcal{F}(x)|\geq 4$,then $|\mathcal{G}(\overline{x})|\leq 11-|\mathcal{F}(x)| $ with equality holding only if  $|\mathcal{F}(x)|= 4 $ or $9$.
\end{claim}
	To prove Claim 5.3, we begin by proving
\begin{lemma}\label{} Let $\mathcal{H}$ be a graph with $\Delta(\mathcal{H})\leq 2$ , $S_1(\mathcal{H})=|\mathcal{H}|+N_1(\mathcal{H})$,
	 $S_2(\mathcal{H})=|\mathcal{H}|+3N_1(\mathcal{H})+N_2(\mathcal{H})$ and $S_3(\mathcal{H})=(|\mathcal{H}|+3N_2(\mathcal{H})+N_3(\mathcal{H}))\mathbf{1}_{|\mathcal{H}|\geq 4}$ , where $\mathbf{1}_{|\mathcal{H}|\geq 4}$ is assigned $1$ if $|\mathcal{H}|\geq 4$ , and  $0$ otherwise.
	 
	(1) If $\tau(\mathcal{H})=1$, then $S_1(\mathcal{H})\leq 3$ and $S_2(\mathcal{H})\leq 7$;
	
	(2) If $\tau(\mathcal{H})=2$, then  $S_2(\mathcal{H})\leq 6$ and $S_3(\mathcal{H})\leq 10$;
	
	(3) If $\tau(\mathcal{H})=3$, then $S_3(\mathcal{H})\leq 10$.
	
\end{lemma}

\begin{proof} Since $\Delta(\mathcal{H})\leq 2$ , each component of $\mathcal{H}$ is a cycle $C_m$ or a path $P_m$. Let $\mathcal{H}_i,1\leq i\leq s$ be all components of $\mathcal{H}$, then $\tau(\mathcal{H})= \displaystyle \sum_{i=1}^s \tau(\mathcal{H}_i)$, while $\tau(\mathcal{H}_i)=\lceil \frac{|\mathcal{H}_i|}{2} \rceil$.
	
	(1) Suppose $\tau(\mathcal{H})=1$ , we may assume $\mathcal{H}=P_1$ or $P_2$, the conclusion follows.
	
	(2) Suppose $\tau(\mathcal{H})=2$ , we may assume $\mathcal{H}=C_4$ or $C_3$ or $P_4$ or $P_3$ or $P_2+P_2$ or $P_2+P_1$ or $P_1+P_1$. The corresponding values of $(N_2(\mathcal{H}),N_3(\mathcal{H})$ are $(2,0)$ or $(3,0)$ or $(1,3)$ or $(3,0)$ or $(1,2)$ or $(2,2)$ or $(4,0)$, the conclusion is easily checked for the seven cases.
	
	(3) Suppose $\tau(\mathcal{H})=3$ and $|\mathcal{H}|\geq 4$, then $S_3(\mathcal{H})=|\mathcal{H}|+N_3(\mathcal{H})$ and $N_3(\mathcal{H})= \displaystyle \prod_{i=1}^s |\mathcal{T}(\mathcal{H}_i)|$. Since  $|\mathcal{T}(C_m)|\geq |\mathcal{T}(P_m)|(m\geq 3)$ and $|\mathcal{T}(P_{2m-1})|> |\mathcal{T}(P_{2m})|(m\geq 1)$ and $|\mathcal{T}(C_{2m-1})|> |\mathcal{T}(C_{2m})|(m\geq 2)$, we may assume $\mathcal{H}=C_5$ or $C_3+P_1$ or $P_2+P_1+P_1$. The corresponding values of $N_3(\mathcal{H})$ are $5$ or $6$ or $4$, the conclusion follows.
	
\end{proof}

Note that $S_2(\mathcal{H})= 7$ if and only if $\mathcal{H}=P_2$. 

Next, we prove Claim 5.3 :

 \begin{proof}
If $\Delta(\mathcal{F}(x))=2$, by $|\mathcal{F}(x)|\geq 4$, $2\leq \tau(\mathcal{F}(x))\leq 3$. Using Lemmas 2.8 and 5.4,  $|\mathcal{G}(\overline{x})|+|\mathcal{F}(x)|\leq S_3(\mathcal{F}(x))\leq 10$.

If  $\Delta(\mathcal{F}(x))=3$, we may assume $\{12,13,14\}\subset \mathcal{F}(x)$ , then $\mathcal{G}(\overline{\{x,1\}})=\{234\}$.

Let us distinguish three cases.

\textbf{Case1}  If there are $y\neq z \in V(\mathcal{F}(x))-[4]$ such that $d(y)=d(z)=3$, then by $\tau(\mathcal{F}(x))\leq 3$ and  $\Delta(\mathcal{F}(x))=3$, $\mathcal{F}(x)=\{1,y,z\}\times \{2,3,4\}$ and $\mathcal{G}(\overline{x})=\{234,1yz\}$, $|\mathcal{G}(\overline{x})|+|\mathcal{F}(x)|=11$ follows.
.

\textbf{Case2} If there is only one $y \in V(\mathcal{F}(x))-[4]$ such that $d(y)=3$, then $\{1,y\}\times \{2,3,4\}\subset \mathcal{F}(x)$. Considering the new graph $\mathcal{H}=\mathcal{F}(x)-\{1,y\}$, if $\mathcal{H}=\emptyset$, then $(N_2(\mathcal{F}(x)),N_3(\mathcal{F}(x)))=(1,1)$, using Lemma 2.8,  $|\mathcal{G}(\overline{x})|+|\mathcal{F}(x)|\leq 3+1+6\leq 10$; if $\mathcal{H}\neq \emptyset$, then for each  $B\neq \{234\}\in  \mathcal{G}(\overline{x})$, $\{1,y\}\subset B$ and $B-\{1,y\}$ is a transversal of $\mathcal{H}$, $\tau(\mathcal{H})=1$ follows. Using Lemma 5.4, $S_1(\mathcal{H})\leq 3$ whence $|\mathcal{G}(\overline{x})|+|\mathcal{F}(x)|\leq (1+N_1(\mathcal{H}))+(6+|\mathcal{H})|)\leq 10$.

\textbf{Case3}  If for all  $y \in V(\mathcal{F}(x))-[4]$ , $d(y)\leq 2$, then considering the new graph $\mathcal{H}=\mathcal{F}(x)-\{1\}$, by  $|\mathcal{F}(x)|\geq 4$ , $\mathcal{H}\neq \emptyset$. For each  $B\neq \{234\}\in  \mathcal{G}(\overline{x})$, $1\in B$ and $B-\{1\}$ is a transversal of $\mathcal{H}$, $\tau(\mathcal{H})\leq 2$ follows.Using Lemma 5.4, $S_2(\mathcal{H})\leq 7$, then $|\mathcal{G}(\overline{x})|+|\mathcal{F}(x)|\leq (1+3N_1(\mathcal{H})+N_2(\mathcal{H}))+(3+|\mathcal{H}|)\leq 11$ with the equality holding only if $\mathcal{H}=P_2$. So we are done .

\end{proof}		
		
	    Now we are able to prove Theorem 1.7:
	    
	  \begin{proof}
We may assume that $|\mathcal{F}|\geq |\mathcal{G}|$ and $|\mathcal{F}(x)|=\Delta(\mathcal{F})$ . Let  $f(1)=f(2)=\infty ,f(3)=10,f(4)=7,f(9)=2,f(k)=10-k, 4<k<9$ and $m=|\mathcal{F}(x)|, n=|\mathcal{G}(x)|$ , then using Claim 2.4 ,5.2 and 5.3, $|\mathcal{F}|\leq 3m$ , $|\mathcal{F}(\overline{x})|\leq f(n)$ and $|\mathcal{G}(\overline{x})|\leq f(m)$. Hence, $|\mathcal{F}||\mathcal{G}|\leq \min\{3m,m+f(n)\}\times \min\{3m,n+f(m)\} =:g(m,n)$.

\begin{claim}\label{}
	$$g(m,n)\leq 121.$$
\end{claim}

  \begin{proof}
  	If $m\leq 3$, then $g(m,n)\leq (3m)^2 \leq 81$.
  	
  	If  $m\geq 4$ and $ n\geq 4$, then   $g(m,n)\leq (m+f(n))(n+f(m))\leq (\frac{22}{2})^2\leq 121$.
  	
  	If $m>4$ and $ n=3$, then $g(m,n)\leq (m+f(n))(n+f(m))\leq 120$.
  
  	In other case, we have $g(m,n)\leq 3m(n+f(m)) \leq 120$.

  	\end{proof}

In summary, $|\mathcal{F}||\mathcal{G}|\leq 121$ with equality holding only if $(|\mathcal{F}(x)|,|\mathcal{G}(\overline{x})|,|\mathcal{G}(x)|,|\mathcal{F}(\overline{x})|)=(9,2,9,2)$ or $(4,7,4,7)$. If the latter, then by proof of Claim 5.3 , $\mathcal{G}(x)=\{12,13,14,25\}$ or $\{12,13,14,34\}$, hence $|\mathcal{F}(1)|\geq |\mathcal{F}(\overline{x})|-|\mathcal{F}(\overline{1})|=6$. This contradicts the maximality of $|\mathcal{F}(x)|$. So must be the former, then by proof of Claim 5.3 , $\mathcal{F}(x)=\{a,b,c\}\times \{y,z,w\},\mathcal{G}(\overline{x})=\{abc,yzw\}$ and $\mathcal{G}(x)=\{a',b',c'\}\times \{y',z',w'\},\mathcal{G}(\overline{x})=\{a'b'c',y'z'w'\}$, since $\mathcal{F}(\overline{x})$ and $\mathcal{G}(\overline{x})$ are cross-intersecting, we may assume $a'=a,b'=y,y'=b,z'=z$, then $\mathcal{F}=\{x\}\times \{a,b,c\}\times \{y,z,w\}\cup \{ayc',bzw'\}$ and 
$\mathcal{G}=\{x\}\times \{a,y,c'\}\times \{b,z,w'\}\cup \{abc,yzw\}$ .	Based on the size of $\{c,w\}\cap \{c',w'\}$, there are three types of $(\mathcal{F},\mathcal{G})$ up to isomorphism.
		
\end{proof}
		
\noindent\textbf{Remark 5.6.} A maximal pair for which the product of their sizes large than 100 is not necessarily optimal, such as
$\mathcal{F}=\{125,136,146,156,256,346,347,356\}$ and  $\mathcal{G}=\{123,135,136,145,146,167,236,246,267,345,356,456,567\}$ . 
		
	\section{Proof of Theorem 1.8}
	
	First, we prove $m(2,2)=9$ :
	
	Let $\mathcal{F}$ and $ \mathcal{G}$ be graphs forming an optimal pair, using (2.1) and Lemma 2.6, $\Delta(\mathcal{F})= 2$. Using Claim 2.3 and  Lemma 5.4,  $|\mathcal{F}|+|\mathcal{G}|=S_2(\mathcal{F})\leq 6$. Now $m(2,2)\leq 9$ follows. By the proof of Lemma 5.4, $\mathcal{F}=\{12,23,13\} =\mathcal{G}$ or $\mathcal{F}=\{12,23,34\}, \mathcal{G}=\{13,24,23\}$ are the only cases attaining $|\mathcal{F}||\mathcal{G}|=9$. Thus these are the only choices for optimal pairs.

	Next, we prove $m(3,2)=25$ :

	Let the 3-graph $\mathcal{F}$ and the graph $ \mathcal{G}$ form an optimal pair and  $|\mathcal{F}(x)|=\Delta(\mathcal{F})\geq 2$ .
	\begin{claim}\label{}
	\[
|\mathcal{F}(x)||\mathcal{G}(\overline{x})|\leq 9.
	\]
\end{claim}

\begin{proof}
	Clearly , $\tau(\mathcal{F}(x))\leq 2$ and $\tau(\mathcal{G}(\overline{x}))=2$, hence $|\mathcal{G}(\overline{x})|\geq 2$.
	
	If $\tau(\mathcal{F}(x))=2$, as $\mathcal{F}(x)$ and $\mathcal{G}(\overline{x})$ are cross-intersecting, the claim follows.
	
	If $\tau(\mathcal{F}(x))=1$, let $\{y\}$ be a transversal of $\mathcal{F}(x)$, then $|\mathcal{F}(x)|=|\mathcal{F}(\{x,y\})|$. Using Lemma 2.6, $|\mathcal{F}(\{x,y\})|\leq 2$ whence $|\mathcal{F}(x)|=2$. So $\mathcal{F}(x)=\{y1,y2\}$, hence$(N_s(\mathcal{F}(x)))_{s=1}^{2}=(1,1)$ , using Lemma 2.8 , $|\mathcal{G}(\overline{x})|\leq 4$.
\end{proof}

Since $\mathcal{G}(x)$ and $\mathcal{F}(\overline{x})$ are  cross-intersecting and every member in $\mathcal{G}(x)$ of size 1,  $\tau(\mathcal{F}(\overline{x}))=1$ and  for all $A\in \mathcal{F}(\overline{x}), \bigcup{\mathcal{G}(x)}\subset A$ follows. Hence $|\mathcal{F}(\overline{x})|\leq |\mathcal{F}(x)|$ and  $\mathcal{F}(\overline{x})\subset \mathcal{F}( \widehat{\bigcup{\mathcal{G}(x)}})$ . Suppose $|\mathcal{G}(x)|=t,1\leq t\leq 3$, then using Lemma 2.6 , $|\mathcal{F}(\overline{x})|\leq 2^{3-t}$. 

If $2\leq t\leq 3$, then $|\mathcal{G}(x)|+|\mathcal{F}(\overline{x})|\leq  t+2^{3-t}\leq 4$  . By Claim 6.1, $|\mathcal{F}(x)|+|\mathcal{G}(\overline{x})|\leq 6$. Hence  $|\mathcal{F}||\mathcal{G}|=(|\mathcal{F}(x)|+|\mathcal{F}(\overline{x})|)(|\mathcal{G}(x)|+|\mathcal{G}(\overline{x})|)\leq (\frac{10}{2})^2\leq 25$.

If $t=1$, then $|\mathcal{F}||\mathcal{G}|\leq 2|\mathcal{F}(x)|(1+|\mathcal{G}(\overline{x})|)\leq 24<25$.

In summary, $|\mathcal{F}||\mathcal{G}|\leq 25$ . Apart from the construction in Theorem 1.5, the equality holds if $\mathcal{F}=\{123,124,x34,x35,x45\}, \mathcal{G}=\{x1,x2,34,35,45\}$ .

Finally, we prove Theorem 1.8 by induction on $k$ :

	When $k=2 $ or $3$, we are done. Suppose $k\geq 4$ and the theorem holds for all $<k$.
	Let the $k$-graph $\mathcal{F}$ and the graph $ \mathcal{G}$ form an optimal pair, and  $|\mathcal{F}(x)|=\Delta(\mathcal{F})\geq 2$ .
    
    \begin{lemma}\label{}  
    	\[
    	|\mathcal{F}(x)||\mathcal{G}(\overline{x})|\leq m(k-1,2).
    	\]
    \end{lemma}
    \begin{proof}
	Clearly , $\tau(\mathcal{F}(x))\leq 2$ and $\tau(\mathcal{G}(\overline{x}))=k-1$.
	Since $|\mathcal{F}(x)|\geq 2$ , we can choose a minimal $S\subset X$ such that $x\in S$ and $\tau(\mathcal{F}(S))= 2$, then $1\leq s=|S|\leq k-1$. Using Lemma 2.6, $|\mathcal{F}(x)|=|\mathcal{F}(S)|\leq 2^{k-s}$. Since $\mathcal{F}(S)$ and $\mathcal{G}(\overline{S})$ are cross-intersecting, $|\mathcal{F}(S)| |\mathcal{G}(\overline{S})|\leq m(k-s,2)$. On the other hand,  $|\mathcal{G}(\overline{x})|\leq (s-1)k +|\mathcal{G}(\overline{S})| $. So 	$|\mathcal{F}(x)||\mathcal{G}(\overline{x})|\leq m(k-s,2)+ 2^{k-s}(s-1)k\leq 1+(k-s)2^{k-s}+2^{k-s}(s-1)k=1+2^{k-s}(k-1)s\leq 1+(k-1)2^{k-1}$ . If the equality holds , then $|\mathcal{F}(S)| |\mathcal{G}(\overline{S})|= m(k-s,2)=2^{k-s}+1$ , $k-s=2$ or $3$ and $|\mathcal{F}(S)|$ is odd follows, hence $s=1$ , the conclusion follows. If the equality does not hold , then using induction hypothesis, 	$|\mathcal{F}(x)||\mathcal{G}(\overline{x})|\leq (k-1)2^{k-1} \leq m(k-1,2)$.

\end{proof}
Since $\mathcal{G}(x)$ and $\mathcal{F}(\overline{x})$ are  cross-intersecting and every member in $\mathcal{G}(x)$ of size 1,  $\tau(\mathcal{F}(\overline{x}))=1$ and  for all $A\in \mathcal{F}(\overline{x}), \bigcup{\mathcal{G}(x)}\subset A$ follows. Hence $|\mathcal{F}(\overline{x})|\leq |\mathcal{F}(x)|$ and  $\mathcal{F}(\overline{x})\subset \mathcal{F}( \widehat{\bigcup{\mathcal{G}(x)}})$ . Suppose $|\mathcal{G}(x)|=t,1\leq t\leq k$, then using Lemma 2.6 , $|\mathcal{F}(\overline{x})|\leq 2^{k-t}$. So
 
 $$|\mathcal{F}||\mathcal{G}|\leq (|\mathcal{F}(x)|+\min\{|\mathcal{F}(x)|,2^{k-t}\})(t+|\mathcal{G}(\overline{x})|).$$ 
Let $m=|\mathcal{F}(x)|,n=|\mathcal{G}(\overline{x})|$, then by Lemma 6.2 and $\tau(\mathcal{G}(\overline{x}))=k-1$, we have  $mn\leq m(k-1,2)$ and $n\geq k-1$.

If $m\leq 2^{k-t}$, then $|\mathcal{F}||\mathcal{G}|\leq 2m(t+n)=2(mn+mt)\leq 2(1+(k-1)2^{k-1}+t2^{k-t}-1 )\leq  k2^k$. The second inequality holds because $mn=m(k-1,2)=1+(k-1)2^{k-1}$ and $m=2^{k-t}$ cannot hold at the same time.

If $m\geq 2^{k-t}$, then $k-1\leq n\leq (k-1)2^{t-1}$ and $|\mathcal{F}||\mathcal{G}|\leq (m+2^{k-t})(t+n)\leq t2^{k-t}+tm+2^{k-t}n+m(k-1,2)=:f(m,n)$.
 	\begin{claim}\label{}
 		
 	$f(m,n)_{max}=\max\{f(2^{k-t},(k-1)2^{t-1}),f(2^{k-1},k-1)\}$.
 	\end{claim}
\begin{proof}
	If $mn\leq (k-1)2^{k-1}$, then  $tm+2^{k-t}n\leq t\frac{(k-1)2^{k-1}}{n} +2^{k-t}n=:g(n)\leq \max\{g(k-1),g((k-1)2^{t-1})\}=\max\{f(2^{k-t},(k-1)2^{t-1}),f(2^{k-1},k-1)\}$
	
	If  $mn=m(k-1,2)=(k-1)2^{k-1}+1$, then by induction hypothesis, $k=4$ and $m=n=5$ and since $n\leq (k-1)2^{t-1}$, $t\in\{2,3,4\}$ follows. If $t=2$ , $f(5,5)<f(4,6)=f(2^{k-t},(k-1)2^{t-1})$; If $t=3$ or 4 , $f(5,5)<f(8,3)=f(2^{k-1},k-1)$; So the claim holds.
\end{proof}

While $f(2^{k-t},(k-1)2^{t-1})=t2^{k-t+1}+(k-1)2^{k}+1\leq k2^{k}+1$ with the equality holding if and only if $t=1$ or $2$.

	\begin{claim}\label{}
$f(2^{k-1},k-1)=(t+k-1)(2^{k-t}+2^{k-1})+1\leq k2^{k}+1$ .
\end{claim}
\begin{proof}
	If $t=k$, then $f(2^{k-1},k-1)=(2k-1)(1+2^{k-1})+1\leq k2^{k}$. If $t\leq k-1$, then $k2^{k}-(t+k-1)(2^{k-t}+2^{k-1})=2^{k-1}(k-t+1)-2^{k-t}(t+k-1)=2^{k-t}(2^{t-1}(k-t+1)-(t+k-1))$, while $2^{t-1}\geq t$, $2^{t-1}(k-t+1)-(t+k-1)\geq t(k-t+1)-(t+k-1)=tk-t^2-k+1=(k-t-1)(t-1)\geq 0$.
\end{proof}

	So  $|\mathcal{F}||\mathcal{G}|\leq k2^{k}+1$. The equality does not hold , suppose the contrary, then $mn=m(k-1,2)=1+(k-1)2^{k-1}$ and $t\neq k$, hence $m$ are both odd and even, contradiction. Thus, $|\mathcal{F}||\mathcal{G}|\leq k2^{k}$.

In summary, $|\mathcal{F}||\mathcal{G}|\leq k2^k$ and the equality holds if $\mathcal{G}=\{12,34,\dots,(2k-1)(2k)\}$ , $\mathcal{F}=\{1,2\}\times \{3,4\}\times \dots \times \{2k-1,2k\}$.

	\section{Concluding remark}
	
	In the present paper we proved that $m(2,2)=3^2$ and $ m(3,3)=11^2$ . For $m(2,2)$, there are two non-isomorphic optimal pairs; For $m(3,3)$, there are three non-isomorphic optimal pairs. But all these pairs can be obtained from the construction in Theorem 1.5. Together with (1.8) for large values of $k$ it seems natural to make the following conjecture.
	\begin{conjecture} For all $k\geq 2$,
		$$m(k,k)=(k^{k-1}+k-1)^2.$$
		\end{conjecture}

	The authors are indebted to Jian Wang for helpful comments concerning Theorem 1.7. He pointed out that some results of [6] can be used to prove $|\mathcal{F}|+|\mathcal{G}| \leq  23$ in the case  $\mathcal{F} \cap \mathcal{G} \neq  \emptyset$. We succeded in refining the argument for this case to yield the desired upper bound 121. However to settle the case  $\mathcal{F} \cap \mathcal{G} =  \emptyset$ eventually we found an argument encompassing both cases.
	
	Finally, let us conclude the paper with another conjecture.
		\begin{conjecture} Suppose that  $\mathcal{F}$ and  $\mathcal{G}$ are cross-intersecting 3-graphs with $\tau(\mathcal{F})=\tau(\mathcal{G})=3$. If $\mathcal{F} \cap \mathcal{G} =  \emptyset$, then $(|\mathcal{F}||\mathcal{G}|)_{\max}=100$.
	\end{conjecture}

Note that there exist examples for which the product of their sizes is 100 . Probably the simplest examples are formed by partitioning the twenty
$3$-subsets of $[6]$.
The twenty $3$-subsets of $[6]$ are naturally partitioned into ten pairs of 
complementary triples $(A_i,B_i)$ . 
Choose five of the $(A_i,B_i)$ to form the first family.
Then form the second family by the remaining five pairs.
These are two disjoint cross-intersecting families, both of size $5\times 2=10$. And in many cases both have covering number three. We also note that it is easy to construct examples with more than 6 vertices from this, such as $\mathcal{F}_1=\{712\}\cup \{345,126\}\cup \{123,456\}\cup \{134,256\} \cup \{135,246\}\cup\{346\}$, $\mathcal{G}_1=\{736\}\cup \{136,245\}\cup \{145,236\}\cup \{146,235\} \cup \{156,234\}\cup\{124\}$ .

	\noindent \textbf{\large{Acknowledgement}}
	
	The authors are indebted to David Avis and Kenichi Kawarabayashi for organizing the WOGGA Workshop at Kyoto University . Discussions at this meeting were the starting point for the research leading to the present paper.

\end{document}